\documentclass[12pt]{amsart}
\usepackage{amsmath,amsthm,amsfonts,amssymb,mathrsfs}
\date{\today}

\usepackage[cp1251]{inputenc}         % Кодова сторінка Windows1251

\usepackage[ukrainian]{babel} % Підтримка кирилиці
\usepackage{amsmath,amsthm,amsfonts,amssymb,mathrsfs}
\usepackage{eucal}

\input xymatrix
\xyoption{all}

\usepackage{color}

\usepackage{amssymb,cite}
\usepackage[colorlinks,plainpages,citecolor=magenta, linkcolor=blue, backref]{hyperref}%,bookmarksnumbered,

\usepackage{hyperref}

\usepackage{tikz}

\newtheorem{theorem}{Теорема}%[section]

\newtheorem{proposition}{Твердження}
\newtheorem{corollary}{Наслiдок}
\newtheorem{lemma}{Лема}

\theoremstyle{definition}
\newtheorem{remark}{Зауваження}%[section]

\newtheorem{definition}{Означення}%[section]

\begin{document}

\title[Про напiвгрупу ендоморфiзмiв напiвгрупи $\boldsymbol{B}_{\omega}^{\mathscr{F}^2}$]{Про напiвгрупу ендоморфiзмiв напiвгрупи $\boldsymbol{B}_{\omega}^{\mathscr{F}^2}$ з двоелементною сiм'Єю $\mathscr{F}^2$ iндуктивних непорожнiх пiдмножин у $\omega$}

\author[Олег~Гутік, Марко Серівка]{Олег~Гутік, Марко Серівка}
\address{Львівський національний університет ім. Івана Франка, Університецька 1, Львів, 79000, Україна}
\email{oleg.gutik@lnu.edu.ua, marko.serivka@lnu.edu.ua}



\begin{abstract}
Ми досліджуємо моноїд $\overline{\boldsymbol{End}}(\boldsymbol{B}_{\omega}^{\mathscr{F}^2})$ усіх ендоморфізмів біциклічного розширення $\boldsymbol{B}_{\omega}^{\mathscr{F}^2}$ з двоелементною сім'єю $\mathscr{F}^2$ індуктивних непорожніх підмножин у $\omega$.  Знайдено підмоноїд $\left\langle\varpi\right\rangle^1$ у напівгрупі $\overline{\boldsymbol{End}}(\boldsymbol{B}_{\omega}^{\mathscr{F}^2})$ такий, що кожен елемент напівгрупи $\overline{\boldsymbol{End}}(\boldsymbol{B}_{\omega}^{\mathscr{F}^2})$ однозначно зображається у вигляді  добутку моноїдального ендоморфізму напівгрупи $\boldsymbol{B}_{\omega}^{\mathscr{F}^2}$ та елемента з $\left\langle\varpi\right\rangle^1$.

УДК 512.53

\emph{Ключові слова та фрази:} Інверсна напівгрупа, біциклічний моноїд, ендоморфізм, біциклічне розширення.

\bigskip
\noindent
\emph{Oleg Gutik, Marko Serivka, \textbf{On the semigroup of endomorphisms of the semigroup $\boldsymbol{B}_{\omega}^{\mathscr{F}^2}$ with the two-element family $\mathscr{F}^2$ of inductive nonempty subsets of $\omega$}.}

\smallskip
\noindent
We study the semigroup $\overline{\boldsymbol{End}}(\boldsymbol{B}_{\omega}^{\mathscr{F}^2})$ of all endomorphisms of the bicyclic extension $\boldsymbol{B}_{\omega}^{\mathscr{F}^2}$ with the two-element family $\mathscr{F}^2$ of inductive nonempty subsets of $\omega$. The submonoid  $\left\langle\varpi\right\rangle^1$ of $\overline{\boldsymbol{End}}(\boldsymbol{B}_{\omega}^{\mathscr{F}^2})$ with the property that every element of the semigroup $\overline{\boldsymbol{End}}(\boldsymbol{B}_{\omega}^{\mathscr{F}^2})$ has the unique representation as the product of the monoid endomorphism of $\boldsymbol{B}_{\omega}^{\mathscr{F}^2}$ and the element of $\left\langle\varpi\right\rangle^1$ is constructed.

\end{abstract}

\maketitle

%\section{Термінологія та означення}

\section{Вступ}\label{section-1}

У цій праці ми користуємося термінологією з монографій \cite{Clifford-Preston-1961, Clifford-Preston-1967, Lawson-1998, Petrich-1984}. Надалі у тексті множину невід'ємних цілих чисел  позначатимемо через $\omega$. Для довільного числа $k\in\omega$ позначимо
${[k)=\{i\in\omega\colon i\geqslant k\}.}$

Нехай $\mathscr{P}(\omega)$~--- сім'я усіх підмножин у $\omega$.
Для довільних $F\in\mathscr{P}(\omega)$ i $n,m\in\omega$ приймемо
\begin{equation*}
n-m+F=\{n-m+k\colon k\in F\}, \qquad \hbox{якщо} \; F\neq\varnothing
\end{equation*}
i $n-m+F=\varnothing$, якщо $F=\varnothing$.
Будемо говорити, що непорожня підсім'я  $\mathscr{F}\subseteq\mathscr{P}(\omega)$ є \emph{${\omega}$-замкненою}, якщо $F_1\cap(-n+F_2)\in\mathscr{F}$ для довільних $n\in\omega$ та $F_1,F_2\in\mathscr{F}$.

Підмножина $A$ в $\omega$ називається \emph{індуктивною}, якщо з  $i\in A$ випливає, що $i+1\in A$. Очевидно, що $\varnothing$ --- індуктивна підмножина в $\omega$.

\begin{remark}\label{remark-2.2}
\begin{enumerate}
  \item\label{remark-2.2(1)} За лемою 6 з \cite{Gutik-Mykhalenych-2020} непорожня множина $F\subseteq \omega$ є індуктивною в $\omega$ тоді і лише тоді, коли $(-1+F)\cap F=F$.
  \item\label{remark-2.2(2)} Оскільки множина $\omega$ зі звичайним порядком є цілком впорядкованою, то для кожної непорожньої індуктивної підмножини $F$ у $\omega$ існує невід'ємне ціле число $n_F\in\omega$ таке, що $[n_F)=F$.
  \item\label{remark-2.2(3)} З \eqref{remark-2.2(2)} випливає, що перетин довільної скінченної кількості непорожніх індуктивних підмножин у $\omega$ є непорожньою індуктивною підмножиною в $\omega$.
\end{enumerate}
\end{remark}

Надалі через $E(S)$ позначатимемо множину ідемпотентів напівгрупи $S$. Напівгрупа ідемпотентів називається \emph{в'язкою}.

Якщо $S$~--- напівгрупа, то на $E(S)$ визначено частковий порядок:
$
e\preccurlyeq f
$   тоді і лише тоді, коли
$ef=fe=e$.
Так означений частковий порядок на $E(S)$ називається \emph{при\-род\-ним}.

Напівгрупа $S$ називається \emph{інверсною}, якщо для довільного елемента $s\in S$ існує єдиний елемент $s^{-1}\in S$ такий, що $ss^{-1}s=s$ i $s^{-1}ss^{-1}=s^{-1}$~\cite{Wagner-1952}. В інверсній напівгрупі $S$ вище означений елемент $s^{-1}$ називається \emph{інверсним до} $s$.

Означимо відношення $\preccurlyeq$ на інверсній напівгрупі $S$ так:
$
    s\preccurlyeq t
$
тоді і лише тоді, коли $s=te$, для деякого ідемпотента $e\in S$. Так означений частковий порядок назива\-єть\-ся \emph{при\-род\-ним част\-ковим порядком} на інверсній напівгрупі $S$~\cite{Wagner-1952}. Очевидно, що звуження природного часткового порядку $\preccurlyeq$ на інверсній напівгрупі $S$ на її в'язку $E(S)$ є при\-род\-ним частковим порядком на $E(S)$.

Якщо $S$~--- напівгрупа, то кожен гомоморфізм (ізоморфізм) з $S$ у $S$ називається \emph{ендо\-мор\-фіз\-мом} (\emph{автоморфізмом}) напівгрупи $S$. Добре відомо, що множина всіх ендоморфізмів фіксованої напівгрупи $S$ (автоморфізмів) стосовно операції композиції відображень є моноїдом (групою). Надалі через $(s)\varepsilon$ i $(A)\varepsilon$ будемо позначати образ елемента $s$ і підмножини $A$ напівгрупи $S$, відповідно, стосовно ендоморфізму $\varepsilon\colon S\to S$. Ендоморфізм $\varepsilon$ напівгрупи $S$ називається \emph{анулюючим}, якщо існує елемент $x\in S$ такий, що $(s)\varepsilon=x$ для всіх $s\in S$. Очевидно, якщо $\varepsilon$ --- анулюючий ендоморфізм напівгрупи $S$, то образ $(S)\varepsilon=x$ є ідемпотентом у $S$.

Нагадаємо (див.  \cite[\S1.12]{Clifford-Preston-1961}), що \emph{біциклічною напівгрупою} (або \emph{біциклічним моноїдом}) ${\mathscr{C}}(p,q)$ називається напівгрупа з одиницею, породжена двоелементною мно\-жи\-ною $\{p,q\}$ і визначена одним  співвідношенням $pq=1$. Біциклічна на\-пів\-група відіграє важливу роль у теорії на\-півгруп. Так, зокрема, класична теорема О.~Ан\-дерсена \cite{Andersen-1952}  стверджує, що {($0$-)}прос\-та напівгрупа з (ненульовим) ідем\-по\-тен\-том є цілком {($0$-)}прос\-тою тоді і лише тоді, коли вона не містить ізоморфну копію бі\-циклічної напівгрупи. Різні розширення та узагальнення біциклічного моноїда вводилися раніше багатьма авторами \cite{Fortunatov-1976, Fotedar-1974, Fotedar-1978, Gutik-Pagon-Pavlyk=2011, Warne-1967}. Такими, зокрема, є конструкції Брука та Брука--Рейлі занурення напівгруп у прості та описання інверсних біпростих і $0$-біпростих $\omega$-напівгруп \cite{Bruck-1958, Reilly-1966, Warne-1966}.

\begin{remark}\label{remark-10}
Легко бачити, що біциклічний моноїд ${\mathscr{C}}(p,q)$ ізоморфний напівгрупі, заданій на множині $\boldsymbol{B}_{\omega}=\omega\times\omega$ з напівгруповою операцією
\begin{equation*}
  (i_1,j_1)\cdot(i_2,j_2)=(i_1+i_2-\min\{j_1,i_2\},j_1+j_2-\min\{j_1,i_2\})=
\left\{
  \begin{array}{ll}
    (i_1-j_1+i_2,j_2), & \hbox{якщо~} j_1\leqslant i_2;\\
    (i_1,j_1-i_2+j_2), & \hbox{якщо~} j_1\geqslant i_2.
  \end{array}
\right.
\end{equation*}
Надалі ми будемо ототожнювати біциклічний моноїд ${\mathscr{C}}(p,q)$ з напівгрупою $\boldsymbol{B}_{\omega}$.
\end{remark}

У праці \cite{Gutik-Mykhalenych-2020} введено алгебраїчні розширення $\boldsymbol{B}_{\omega}^{\mathscr{F}}$ біциклічного моноїда для довільної $\omega$-замк\-не\-ної сім'ї $\mathscr{F}$ підмножин в $\omega$, які узагальнюють біциклічний моноїд, зліченну напівгрупу матричних одиниць і деякі інші комбінаторні інверсні напівгрупи.

Нагадаємо цю конструкцію.
Нехай $\boldsymbol{B}_{\omega}$~--- біциклічний моноїд і  $\mathscr{F}$ --- непорожня ${\omega}$-замкнена підсім'я в  $\mathscr{P}(\omega)$. На множині $\boldsymbol{B}_{\omega}\times\mathscr{F}$ озна\-чимо бінарну операцію ``$\cdot$''  формулою
\begin{equation}\label{eq-1.1}
  (i_1,j_1,F_1)\cdot(i_2,j_2,F_2)=
  \left\{
    \begin{array}{ll}
      (i_1-j_1+i_2,j_2,(j_1-i_2+F_1)\cap F_2), & \hbox{якщо~} j_1<i_2;\\
      (i_1,j_2,F_1\cap F_2),                   & \hbox{якщо~} j_1=i_2;\\
      (i_1,j_1-i_2+j_2,F_1\cap (i_2-j_1+F_2)), & \hbox{якщо~} j_1>i_2.
    \end{array}
  \right.
\end{equation}

У \cite{Gutik-Mykhalenych-2020} доведено, якщо сім'я  $\mathscr{F}\subseteq\mathscr{P}(\omega)$ є ${\omega}$-замкненою, то $(\boldsymbol{B}_{\omega}\times\mathscr{F},\cdot)$ є напівгрупою, а також, що $\boldsymbol{B}_{\omega}^{\mathscr{F}}$ є комбінаторною інверсною напівгрупою, і описано відношення Ґріна, частковий природний порядок на напівгрупі $\boldsymbol{B}_{\omega}^{\mathscr{F}}$ та її множину ідемпотентів. Крім того, у \cite{Gutik-Mykhalenych-2020} доведено критерії  простоти, $0$-простоти, біпростоти та $0$-біпростоти напівгрупи $\boldsymbol{B}_{\omega}^{\mathscr{F}}$, і вказано умови, коли $\boldsymbol{B}_{\omega}^{\mathscr{F}}$ містить одиницю, ізоморфна біциклічному моноїду або зліченній напівгрупі матричних одиниць.

Припустимо, що ${\omega}$-замкнена сім'я $\mathscr{F}\subseteq\mathscr{P}(\omega)$ містить порожню множину $\varnothing$, то з означення напівгрупової операції ``$\cdot$'' на $\boldsymbol{B}_{\omega}\times\mathscr{F}$ випливає, що множина
$ %\begin{equation*}
  \boldsymbol{I}=\{(i,j,\varnothing)\colon i,j\in\omega\}
$ %\end{equation*}
є ідеалом у $(\boldsymbol{B}_{\omega}\times\mathscr{F},\cdot)$.

\begin{definition}[\!\!{\cite{Gutik-Mykhalenych-2020}}]\label{definition-1.1}
Для довільної ${\omega}$-замкненої сім'ї $\mathscr{F}\subseteq\mathscr{P}(\omega)$ означимо
\begin{equation*}
  \boldsymbol{B}_{\omega}^{\mathscr{F}}=
\left\{
  \begin{array}{ll}
    (\boldsymbol{B}_{\omega}\times\mathscr{F},\cdot)/\boldsymbol{I}, & \hbox{якщо~} \varnothing\in\mathscr{F};\\
    (\boldsymbol{B}_{\omega}\times\mathscr{F},\cdot), & \hbox{якщо~} \varnothing\notin\mathscr{F}.
  \end{array}
\right.
\end{equation*}
\end{definition}

У \cite{Gutik-Lysetska=2021, Lysetska=2020} досліджено алгебраїчну структуру напівгрупи $\boldsymbol{B}_{\omega}^{\mathscr{F}}$ у випадку, коли ${\omega}$-замкнена сім'я $\mathscr{F}$ складається з атомарних підмножин (одноточкових підмножин і порожньої множини) в ${\omega}$. Зокрема доведено, що за виконання таких умов на сім'ю $\mathscr{F}$ напівгрупа  $\boldsymbol{B}_{\omega}^{\mathscr{F}}$ ізоморфна піднапівгрупі розширення Брандта напівгратки $(\omega,\min)$.

Із зауваження~\ref{remark-2.2}\eqref{remark-2.2(3)} випливає, якщо сім'я $\mathscr{F}_0$ склада\-єть\-ся з індуктивних в $\omega$ підмножин і міс\-тить порожню множину $\varnothing$ як елемент, то для сім'ї $\mathscr{F}=\mathscr{F}_0\setminus\{\varnothing\}$ множина  $\boldsymbol{B}_{\omega}^{\mathscr{F}}$ з індукованою напівгруповою операцією з  $\boldsymbol{B}_{\omega}^{\mathscr{F}_0}$ є піднапівгрупою в $\boldsymbol{B}_{\omega}^{\mathscr{F}_0}$.

У \cite{Gutik-Mykhalenych-2021} досліджуються групові конгруенції на напівгрупі $\boldsymbol{B}_{\omega}^{\mathscr{F}}$ та її гомоморфні рет\-рак\-ти у випадку, коли $\mathscr{F}$ --- ${\omega}$-замкнена сім'я  з індуктивних непорожніх підмножин в $\omega$. Доведено, що конгруенція $\mathfrak{C}$ на $\boldsymbol{B}_{\omega}^{\mathscr{F}}$ є груповою, тоді і лише тоді, коли звуження конгруенції $\mathfrak{C}$ на піднапівгрупу в $\boldsymbol{B}_{\omega}^{\mathscr{F}}$, яка ізоморфна біциклічній напівгруп, не є відношенням рівності. Також, у \cite{Gutik-Mykhalenych-2021} описано всі нетривіальні гомоморфні ретракти й ізоморфізми напівгрупи $\boldsymbol{B}_{\omega}^{\mathscr{F}}$.

Ендоморфізми біциклічного моноїда та розширеної біциклічної напівгрупи описані в \cite{Gutik-Prokhorenkova-Sekh-2021}, і отримані результати поширено на біциклічні розширення лінійно впорядкованих архімедових напівгруп у \cite{Gutik-Prokhorenkova-2022}.

Надалі скрізь в тексті ми вважаємо, що ${\omega}$-замкнена сім'я $\mathscr{F}$ складається лише з індуктивних непорожніх підмножин у $\omega$. У \cite{Gutik-Mykhalenych-2022} доведено, що iн'\-cктивний ендоморфізм $\varepsilon$ напiвгрупи $\boldsymbol{B}_{\omega}^{\mathscr{F}}$ є тотожним відображенням тоді і тільки тоді, коли $\varepsilon$ має три різні нерухомі точки, що еквівалентно існуванню неідемпотентного елемента $(i,j,[p))\in\boldsymbol{B}_{\omega}^{\mathscr{F}}$ такого, що $(i,j,[p))\varepsilon=(i,j,[p))$.

У \cite{Gutik-Pozdniakova=2022} досліджено ін'єктивні ендоморфізми напівгрупи $\boldsymbol{B}_{\omega}^{\mathscr{F}^2}$ з двоелементною сім'єю $\mathscr{F}^2$ ін\-дук\-тив\-них непорожніх підмножин у $\omega$. Описано елементи напівгрупи $\boldsymbol{End}^1_*(\boldsymbol{B}_{\omega}^{\mathscr{F}^2})$ усіх ін'єктивних моноїдальних ендоморфізмів, тобто які зберігають одиницю моноїда $\boldsymbol{B}_{\omega}^{\mathscr{F}^2}$, і доведено, що відно\-шен\-ня Ґріна $\mathscr{R}$, $\mathscr{L}$, $\mathscr{H}$, $\mathscr{D}$ і $\mathscr{J}$ на напівгрупі  $\boldsymbol{End}^1_*(\boldsymbol{B}_{\omega}^{\mathscr{F}^2})$ збігаються з відношенням рівності. У \cite{Gutik-Pozdniakova=2023, Gutik-Pozdniakova=2024} досліджено структуру напівгрупи $\boldsymbol{End}(\boldsymbol{B}_{\omega}^{\mathscr{F}^2})$ усіх моноїдальних ендоморфізмів напівгрупи $\boldsymbol{B}_{\omega}^{\mathscr{F}^2}$, напівгрупову операцію та відношення Ґріна на напівгрупі $\boldsymbol{End}(\boldsymbol{B}_{\omega}^{\mathscr{F}^2})$.

У цій праці ми досліджуємо немоноїдальні ендоморфізми (тобто такі, що не зберігають одиницю) напівгрупи $\boldsymbol{B}_{\omega}^{\mathscr{F}^2}$ з двоелементною сім'єю $\mathscr{F}^2$ індуктивних непорожніх підмножин у $\omega$. Надалі через $\overline{\boldsymbol{End}}(\boldsymbol{B}_{\omega}^{\mathscr{F}^2})$ позначатимемо моноїд усіх ендоморфізмів напівгрупи $\boldsymbol{B}_{\omega}^{\mathscr{F}^2}$. Знайдено підмоноїд $\left\langle\varpi\right\rangle^1$ у $\overline{\boldsymbol{End}}(\boldsymbol{B}_{\omega}^{\mathscr{F}^2})$ такий, що кожен елемент напівгрупи $\overline{\boldsymbol{End}}(\boldsymbol{B}_{\omega}^{\mathscr{F}^2})$ однозначно зображається у вигляді  добутку моноїдального ендоморфізму напівгрупи $\boldsymbol{B}_{\omega}^{\mathscr{F}^2}$ та елемента з $\left\langle\varpi\right\rangle^1$.

За твердженням 1 з \cite{Gutik-Mykhalenych-2020}, не зменшуючи загальності, надалі в цьому тексті можемо вважати, що двоелементна сім'я  $\mathscr{F}^2$ збігається з $\{[0), [1)\}$.

%%%%%%%%%%%%%%%%%%%%%%%%%%%%%%%%%%%%%%%%%%%%%%%%%%%%%%%%%%%%%%%%%%%%%%%%%%%%%%%%%%
\section{Про піднапівгрупи в $\boldsymbol{B}_{\omega}^{\mathscr{F}^2}$ та пов'язані з ними ін'єктивні ендоморфізми}\label{section-2}

Для довільного $p\in\{0,1\}$ позначимо $\boldsymbol{B}_{\omega}^{\mathscr{F}^2_p}=\{(i,j,[p))\colon i,j\in\omega\}$.

Означимо відображення $\varpi\colon \boldsymbol{B}_{\omega}^{\mathscr{F}^2}\to \boldsymbol{B}_{\omega}^{\mathscr{F}^2}$ за формулою
\begin{equation*}
  (i,j,[p)))\varpi=
  \left\{
    \begin{array}{ll}
      (i,j,[1))),     & \hbox{якщо~} p=0;\\
      (i+1,j+1,[0))), & \hbox{якщо~} p=1.
    \end{array}
  \right.
\end{equation*}

\begin{lemma}\label{lemma-2.1}
Відображення $\varpi\colon \boldsymbol{B}_{\omega}^{\mathscr{F}^2}\to \boldsymbol{B}_{\omega}^{\mathscr{F}^2}$ є ін'єктивним ендоморфізмом напівгрупи $\boldsymbol{B}_{\omega}^{\mathscr{F}^2}$.
\end{lemma}

\begin{proof}
За твердженням~3 з \cite{Gutik-Mykhalenych-2020} для довільного числа $p\in\{0,1\}$ напівгрупа $\boldsymbol{B}_{\omega}^{\mathscr{F}^2_p}$ ізоморфна бі\-цик\-ліч\-ній напівгрупі. Звідси випливають такі рівності
\begin{equation*}
  ((i_1,j_1,[0))\cdot (i_2,j_2,[0)))\varpi=(i_1,j_1,[0))\varpi\cdot(i_2,j_2,[0))\varpi
\end{equation*}
i
\begin{equation*}
  ((i_1,j_1,[1))\cdot (i_2,j_2,[1)))\varpi=(i_1,j_1,[1))\varpi\cdot(i_2,j_2,[1))\varpi
\end{equation*}
для довільних $i_1,j_1,i_2,j_2\in\omega$, оскільки відображення $\iota\colon \boldsymbol{B}_{\omega}\to \boldsymbol{B}_{\omega}$, $(i,j)\mapsto(i,j)$ --- єдиний автоморфізм біциклічного моноїда $\boldsymbol{B}_{\omega}$.

Для довільних $i_1,j_1,i_2,j_2\in\omega$ маємо, що
\begin{align*}
  ((i_1,j_1,[0))\cdot (i_2,j_2,[1)))\varpi&=
  \left\{
    \begin{array}{ll}
      (i_1-j_1+i_2,j_2,(j_1-i_2+[0))\cap [1))\varpi, & \hbox{якщо~} j_1<i_2;\\
      (i_1,j_2,[0)\cap [1))\varpi,                   & \hbox{якщо~} j_1=i_2;\\
      (i_1,j_1-i_2+j_2,[0)\cap (i_2-j_1+[1)))\varpi, & \hbox{якщо~} j_1>i_2
    \end{array}
  \right.= \\
   &=
   \left\{
    \begin{array}{ll}
      (i_1-j_1+i_2,j_2,[1))\varpi, & \hbox{якщо~} j_1<i_2;\\
      (i_1,j_2,[1))\varpi,         & \hbox{якщо~} j_1=i_2;\\
      (i_1,j_1-i_2+j_2,[0))\varpi, & \hbox{якщо~} j_1>i_2
    \end{array}
  \right.= \\
   &=
   \left\{
    \begin{array}{ll}
      (i_1-j_1+i_2+1,j_2+1,[0)), & \hbox{якщо~} j_1<i_2;\\
      (i_1+1,j_2+1,[0)),         & \hbox{якщо~} j_1=i_2;\\
      (i_1,j_1-i_2+j_2,[1)),     & \hbox{якщо~} j_1>i_2,
    \end{array}
  \right.
\end{align*}
\begin{align*}
  (i_1,j_1,[0))\varpi\cdot(i_2,j_2,[1))\varpi&=(i_1,j_1,[1))\cdot(i_2+1,j_2+1,[0))= \\
   &=
   \left\{
    \begin{array}{ll}
      (i_1-j_1+i_2+1,j_2+1,(j_1-i_2-1+[1))\cap[0)), & \hbox{якщо~} j_1<i_2{+}1;\\
      (i_1,j_2+1,[1)\cap[0)),                       & \hbox{якщо~} j_1=i_2{+}1;\\
      (i_1,j_1-i_2+j_2,(i_2+1-j_1+[0))\cap[1)),     & \hbox{якщо~} j_1>i_2{+}1
    \end{array}
  \right.=\\
   &=
   \left\{
    \begin{array}{ll}
      (i_1-j_1+i_2+1,j_2+1,[0)), & \hbox{якщо~} j_1<i_2+1;\\
      (i_1,j_2+1,[1)),           & \hbox{якщо~} j_1=i_2+1;\\
      (i_1,j_1-i_2+j_2,[1)),     & \hbox{якщо~} j_1>i_2+1
    \end{array}
  \right.=
  \\
   &=
   \left\{
    \begin{array}{ll}
      (i_1-j_1+i_2+1,j_2+1,[0)), & \hbox{якщо~} j_1<i_2;\\
      (i_1+1,j_2+1,[0)),         & \hbox{якщо~} j_1=i_2;\\
      (i_1,j_1-i_2-1+j_2+1,[1)),           & \hbox{якщо~} j_1=i_2+1;\\
      (i_1,j_1-i_2+j_2,[1)),     & \hbox{якщо~} j_1>i_2+1
    \end{array}
  \right.=\\
   &=
   \left\{
    \begin{array}{ll}
      (i_1-j_1+i_2+1,j_2+1,[0)), & \hbox{якщо~} j_1<i_2;\\
      (i_1+1,j_2+1,[0)),         & \hbox{якщо~} j_1=i_2;\\
      (i_1,j_1-i_2+j_2,[1)),     & \hbox{якщо~} j_1>i_2
    \end{array}
  \right.
\end{align*}
i
\begin{align*}
  ((i_1,j_1,[1))\cdot (i_2,j_2,[0)))\varpi&=
   \left\{
    \begin{array}{ll}
      (i_1-j_1+i_2,j_2,(j_1-i_2+[1))\cap [0))\varpi, & \hbox{якщо~} j_1<i_2;\\
      (i_1,j_2,[1)\cap [0))\varpi,                   & \hbox{якщо~} j_1=i_2;\\
      (i_1,j_1-i_2+j_2,[1)\cap (i_2-j_1+[0)))\varpi, & \hbox{якщо~} j_1>i_2
    \end{array}
  \right.= \\
   &= \left\{
    \begin{array}{ll}
      (i_1-j_1+i_2,j_2,[0))\varpi, & \hbox{якщо~} j_1<i_2;\\
      (i_1,j_2,[1))\varpi,         & \hbox{якщо~} j_1=i_2;\\
      (i_1,j_1-i_2+j_2,[1))\varpi, & \hbox{якщо~} j_1>i_2
    \end{array}
  \right.=\\
  &= \left\{
    \begin{array}{ll}
      (i_1-j_1+i_2,j_2,[1)),     & \hbox{якщо~} j_1<i_2;\\
      (i_1+1,j_2+1,[0)),         & \hbox{якщо~} j_1=i_2;\\
      (i_1+1,j_1-i_2+j_2+1,[0)), & \hbox{якщо~} j_1>i_2,
    \end{array}
  \right.
\end{align*}
\begin{align*}
  (i_1,j_1,[1))\varpi\cdot(i_2,j_2,[0))\varpi&=(i_1+1,j_1+1,[0))\cdot(i_2,j_2,[1))= \\
   &=
   \left\{
    \begin{array}{ll}
      (i_1-j_1+i_2,j_2,(j_1+1-i_2+[0))\cap[1)),     & \hbox{якщо~} j_1{+}1<i_2;\\
      (i_1+1,j_2,[0)\cap[1)),                       & \hbox{якщо~} j_1{+}1=i_2;\\
      (i_1+1,j_1-i_2+j_2+1,[0)\cap(i_2-j_1-1+[1))), & \hbox{якщо~} j_1{+}1>i_2
    \end{array}
  \right.=\\
    &=
   \left\{
    \begin{array}{ll}
      (i_1-j_1+i_2,j_2,[1)),     & \hbox{якщо~} j_1+1<i_2;\\
      (i_1+1,j_2,[1)),           & \hbox{якщо~} j_1+1=i_2;\\
      (i_1+1,j_1-i_2+j_2+1,[0)), & \hbox{якщо~} j_1+1>i_2
    \end{array}
  \right.=%\\
  \end{align*}
\begin{align*}  
    &=
   \left\{
    \begin{array}{ll}
      (i_1-j_1+i_2,j_2,[1)),       & \hbox{якщо~} j_1+1<i_2;\\
      (i_1+1-(j_1+1)+i_2,j_2,[1)), & \hbox{якщо~} j_1+1=i_2;\\
      (i_1+1,j_1-i_2+j_2+1,[0)),   & \hbox{якщо~} j_1=i_2;\\
      (i_1+1,j_1-i_2+j_2+1,[0)),   & \hbox{якщо~} j_1>i_2
    \end{array}
  \right.=\\
    &= \left\{
    \begin{array}{ll}
      (i_1-j_1+i_2,j_2,[1)),     & \hbox{якщо~} j_1<i_2;\\
      (i_1+1,j_2+1,[0)),         & \hbox{якщо~} j_1=i_2;\\
      (i_1+1,j_1-i_2+j_2+1,[0)), & \hbox{якщо~} j_1>i_2.
    \end{array}
  \right.
\end{align*}
З вище викладених рівностей випливає, що відображення $\varpi$ є ендоморфізмом напівгрупи $\boldsymbol{B}_{\omega}^{\mathscr{F}^2}$. Ін'єктивність відображення $\varpi$ очевидна.
\end{proof}

\begin{proposition}\label{proposition-2.2}
Для довільного натурального числа $n$ виконуються такі рівності
\begin{align*}
 (i,j,[0))\varpi^{2n}&=(i+n,j+n,[0)); \\
 (i,j,[1))\varpi^{2n}&=(i+n,j+n,[1)); \\
 (i,j,[0))\varpi^{2n-1}&=(i+n-1,j+n-1,[1)); \\
 (i,j,[1))\varpi^{2n-1}&=(i+n,j+n,[0)),
\end{align*}
$i,j\in\omega$.
\end{proposition}

\begin{proof}
З визначення ендоморфізму $\varpi\colon \boldsymbol{B}_{\omega}^{\mathscr{F}^2}\to \boldsymbol{B}_{\omega}^{\mathscr{F}^2}$ випливає, що
\begin{align*}
  (i,j,[0))\varpi^{2}&=(i,j,[1))\varpi=(i+1,j+1,[0)); \\
 (i,j,[1))\varpi^{2} &=(i+1,j+1,[0))\varpi=(i+1,j+1,[1)).
\end{align*}
Далі скористаємося індукцією.
\end{proof}

З леми~\ref{lemma-2.1} і твердження~\ref{proposition-2.2} випливає

\begin{corollary}\label{corollary-2.3}
Для довільного натурального числа $n$ відображення $\varpi^n\colon \boldsymbol{B}_{\omega}^{\mathscr{F}^2}\to \boldsymbol{B}_{\omega}^{\mathscr{F}^2}$ є ін'єктивним ендоморфізмом напівгрупи $\boldsymbol{B}_{\omega}^{\mathscr{F}^2}$.
\end{corollary}

Означивши $(i,j,[p))\varpi^0=(i,j,[p))$, отримуємо, що $\varpi^0$ є одиницею напівгрупи $\overline{\boldsymbol{End}}(\boldsymbol{B}_{\omega}^{\mathscr{F}^2})$, а отже, $\left\langle\varpi\right\rangle=\left\{\varpi^n\colon n\in\mathbb{N}\right\}$  i $\left\langle\varpi\right\rangle^1=\left\langle\varpi\right\rangle\cup\{\varpi^0\}$ є циклічною піднапівгрупою та циклічним підмоноїдом моноїда  $\overline{\boldsymbol{End}}(\boldsymbol{B}_{\omega}^{\mathscr{F}^2})$, відповідно.

Для довільних $s\in\omega$ i $p\in\{0,1\}$ означимо
\begin{equation*}
  \boldsymbol{B}_{\omega}^{\mathscr{F}^2}(s,p)=(s,s,[p))\boldsymbol{B}_{\omega}^{\mathscr{F}^2}(s,s,[p))= (s,s,[p))\boldsymbol{B}_{\omega}^{\mathscr{F}^2}\cap\boldsymbol{B}_{\omega}^{\mathscr{F}^2}(s,s,[p)).
\end{equation*}
З визначення напівгрупової операції на $\boldsymbol{B}_{\omega}^{\mathscr{F}^2}$ випливає, що для довільних $s\in\omega$ i $p\in\{0,1\}$ множина $\boldsymbol{B}_{\omega}^{\mathscr{F}^2}(s,p)$ є піднапівгрупою в $\boldsymbol{B}_{\omega}^{\mathscr{F}^2}$ з одиничним елементом $(s,s,[p))$.

\begin{proposition}\label{proposition-2.4}
Для довільного числа $s\in\omega$  виконуються такі рівності
\begin{equation*}
  \boldsymbol{B}_{\omega}^{\mathscr{F}^2}(s,0)=\big(\boldsymbol{B}_{\omega}^{\mathscr{F}^2}\big)\varpi^{2s} \qquad \hbox{i} \qquad \boldsymbol{B}_{\omega}^{\mathscr{F}^2}(s,1)=\big(\boldsymbol{B}_{\omega}^{\mathscr{F}^2}\big)\varpi^{2s-1}.
\end{equation*}
Більше того, для довільних $s\in\omega$ i $p\in\{0,1\}$ напівгрупа $\boldsymbol{B}_{\omega}^{\mathscr{F}^2}(s,p)$ ізоморфна напівгрупі $\boldsymbol{B}_{\omega}^{\mathscr{F}^2}$.
\end{proposition}

\begin{proof}
Перша частина твердження випливає з твердження~\ref{proposition-2.2}. Далі скористаємося наслідком~\ref{corollary-2.3}.
\end{proof}

%%%%%%%%%%%%%%%%%%%%%%%%%%%%%%%%%%%%%%%%%%%%%%%%%%%%%%%%%%%%%%%%%%%%%%%%%%%%%%%%%%
\section{Зображення ендоморфізмфів напівгрупи $\boldsymbol{B}_{\omega}^{\mathscr{F}^2}$ композицією її моноїдальних ендоморфізмів і елементами напівгрупи $\left\langle\varpi\right\rangle^1$}\label{section-2}

Зафіксуємо довільний ендоморфізм $\varepsilon$ напівгрупи $\boldsymbol{B}_{\omega}^{\mathscr{F}^2}$. За твердженням~1.4.21~\cite{Lawson-1998} i лемою~2~\cite{Gutik-Mykhalenych-2020} існує ідемпотент $(s,s,[p))\in \boldsymbol{B}_{\omega}^{\mathscr{F}^2}$, $s\in\omega$, $p\in\{0,1\}$, такий, що $(s,s,[p))=(0,0,[0))\varepsilon$. Оскільки $(0,0,[0))$~--- одиниця напівгрупи $\boldsymbol{B}_{\omega}^{\mathscr{F}^2}$, то $(s,s,[p))$~--- одиничний елемент напівгрупи $\big(\boldsymbol{B}_{\omega}^{\mathscr{F}^2}\big)\varepsilon$, а отже, образ $\big(\boldsymbol{B}_{\omega}^{\mathscr{F}^2}\big)\varepsilon$ є піднапівгрупою в $\boldsymbol{B}_{\omega}^{\mathscr{F}^2}(s,p)$. За твердженням~\ref{proposition-2.4} існує число $n\in\omega$ таке, що $(s,s,[p))=(0,0,[0))\varpi^n$. З ін'єктивності ендоморфізму $\varpi^n\colon \boldsymbol{B}_{\omega}^{\mathscr{F}^2}\to \boldsymbol{B}_{\omega}^{\mathscr{F}^2}$ (наслідок~\ref{corollary-2.3}) випливає, що композиція $\varepsilon\big(\varpi^n\big)^{-1}$ є моноїдальним ендоморфізмом моноїда $\boldsymbol{B}_{\omega}^{\mathscr{F}^2}$. Позначимо $\varepsilon_1=\varepsilon\big(\varpi^n\big)^{-1}$. З ін'єктивності ендоморфізму $\varpi^n\colon \boldsymbol{B}_{\omega}^{\mathscr{F}^2}\to \boldsymbol{B}_{\omega}^{\mathscr{F}^2}$  випливає, що $\varepsilon=\varepsilon_1\varpi^n$. Оскільки часткове відоб\-раження $\big(\varpi^n\big)^{-1}\colon \boldsymbol{B}_{\omega}^{\mathscr{F}^2}\rightharpoonup \boldsymbol{B}_{\omega}^{\mathscr{F}^2}$ є ін'єктивним і $\big(\boldsymbol{B}_{\omega}^{\mathscr{F}^2}(s,p)\big)\big(\varpi^n\big)^{-1}=\boldsymbol{B}_{\omega}^{\mathscr{F}^2}$, то ендоморфізм $\varepsilon_1$ визначено коректно, і він єдиний.

Отже, доведена така теорема.

\begin{theorem}\label{theorem-3.1}
Для довільного ендоморфізму $\varepsilon$ напівгрупи $\boldsymbol{B}_{\omega}^{\mathscr{F}^2}$ існують єдиний моноїдальний ендоморфізм  $\varepsilon_1\colon \boldsymbol{B}_{\omega}^{\mathscr{F}^2}\to \boldsymbol{B}_{\omega}^{\mathscr{F}^2}$ та єдине невід'ємне ціле число $n$ такі, що $\varepsilon=\varepsilon_1\varpi^n$, причому ендоморфізм $\varepsilon$ є ін'єктивним (анулюючим) тоді і лише тоді, коли  моноїдальний ендоморфізм  $\varepsilon_1$ є ін'єктивним (анулюючим).
\end{theorem}

З теореми~\ref{theorem-3.1} випливає

\begin{corollary}\label{corollary-3.2}
$\overline{\boldsymbol{End}}(\boldsymbol{B}_{\omega}^{\mathscr{F}^2})=\boldsymbol{End}(\boldsymbol{B}_{\omega}^{\mathscr{F}^2})\cdot \left\langle\varpi\right\rangle^1$.
\end{corollary}

Ін'єктивні моноїдальні ендоморфізми напівгрупи $\boldsymbol{B}_{\omega}^{\mathscr{F}^2}$ описані в праці \cite{Gutik-Pozdniakova=2022}. За теоремою~1 з \cite{Gutik-Pozdniakova=2022} кожен такий ендоморфізм $\varepsilon$ є одного з двох типів: або він збігається з перетворенням $\alpha_{k,p}$ моноїда $\boldsymbol{B}_{\omega}^{\mathscr{F}^2}$, яке визначається
\begin{align*}
  (i,j,[0))\alpha_{k,p}&=(ki,kj,[0)), \\
  (i,j,[1))\alpha_{k,p}&=(p+ki,p+kj,[1)),
\end{align*}
для всіх $i,j\in\omega$, де $k$~--- деяке натуральне число і $p\in\{0,\ldots,k-1\}$, або він збігається з перетворенням $\beta_{k,p}$ моноїда $\boldsymbol{B}_{\omega}^{\mathscr{F}^2}$, яке визначається
\begin{align*}
  (i,j,[0))\beta_{k,p}&=(ki,kj,[0)), \\
  (i,j,[1))\beta_{k,p}&=(p+ki,p+kj,[0)),
\end{align*}
для всіх $i,j\in\omega$, де $k$~--- деяке натуральне число $\geqslant 2$ i $p\in\{1,\ldots,k-1\}$.

Якщо ендоморфізм $\varepsilon$ напівгрупи $\boldsymbol{B}_{\omega}^{\mathscr{F}^2}$ є ін'єктивним відображенням, то з ін'єктивності ендоморфізму $\varpi^n\colon \boldsymbol{B}_{\omega}^{\mathscr{F}^2}\to \boldsymbol{B}_{\omega}^{\mathscr{F}^2}$ випливає, що  $\varepsilon_1=\varepsilon\big(\varpi^n\big)^{-1}$ є моноїдальним ін'єктивним ендоморфізмом моноїда $\boldsymbol{B}_{\omega}^{\mathscr{F}^2}$. Отже, виконується така теорема.

\begin{theorem}\label{theorem-3.3}
Для довільного ін'єктивного ендоморфізму $\varepsilon$ напівгрупи $\boldsymbol{B}_{\omega}^{\mathscr{F}^2}$ виконується лише одна з таких умов:
\begin{itemize}
  \item[$(i)$]  існують єдині невід'ємне ціле число $n$,  натуральне число $k$ і $p\in\{0,\ldots,k-1\}$ такі, що $\varepsilon=\alpha_{k,p}\varpi^n$;
  \item[$(ii)$] існують єдині невід'ємне ціле число $n$,  натуральне число $k\geqslant 2$ і $p\in\{1,\ldots,k-1\}$ такі, що $\varepsilon=\beta_{k,p}\varpi^n$.
\end{itemize}
\end{theorem}

Очевидно, що для довільної напівгрупи $S$ композиція її двох (моноїдальних) ін'єктивних ендоморфізмів є (моноїдальним) ін'єктивним ендоморфізмом.

Через $\overline{\boldsymbol{End}_{\mathrm{inj}}}(\boldsymbol{B}_{\omega}^{\mathscr{F}^2})$ i $\boldsymbol{End}_{\mathrm{inj}}(\boldsymbol{B}_{\omega}^{\mathscr{F}^2})$ позначимо моноїди усіх ін'єктивних і всіх моноїдальних ін'єк\-тив\-них ендоморфізмів моноїда  $\boldsymbol{B}_{\omega}^{\mathscr{F}^2}$. Очевидно, що $\boldsymbol{End}_{\mathrm{inj}}(\boldsymbol{B}_{\omega}^{\mathscr{F}^2})$ --- піднапівгрупа в $\overline{\boldsymbol{End}_{\mathrm{inj}}}(\boldsymbol{B}_{\omega}^{\mathscr{F}^2})$ і в $\boldsymbol{End}(\boldsymbol{B}_{\omega}^{\mathscr{F}^2})$, а $\overline{\boldsymbol{End}_{\mathrm{inj}}}(\boldsymbol{B}_{\omega}^{\mathscr{F}^2})$ --- піднапівгрупа в $\overline{\boldsymbol{End}}(\boldsymbol{B}_{\omega}^{\mathscr{F}^2})$.

З теореми~\ref{theorem-3.3} випливає

\begin{corollary}\label{corollary-3.4}
$\overline{\boldsymbol{End}_{\mathrm{inj}}}(\boldsymbol{B}_{\omega}^{\mathscr{F}^2})=\boldsymbol{End}_{\mathrm{inj}}(\boldsymbol{B}_{\omega}^{\mathscr{F}^2}) \cdot \left\langle\varpi\right\rangle^1$.
\end{corollary}

Очевидно, що для довільних $s\in\omega$ i $q\in\{0,1\}$ відображення $\chi_{s,q}\colon \boldsymbol{B}_{\omega}^{\mathscr{F}^2}\to \boldsymbol{B}_{\omega}^{\mathscr{F}^2}$, означене
\begin{equation*}
  (i,j,[p))\chi_{s,q}=(s,s,[q)),
\end{equation*}
для довільних $i,j\in\omega$ і $p\in\{0,1\}$, є анулюючим ендоморфізмом напівгрупи  $\boldsymbol{B}_{\omega}^{\mathscr{F}^2}$. Також, з визначення ендоморфізму  $\chi_{s,q}$ випливає, що для довільного ендоморфізму $\varepsilon\colon \boldsymbol{B}_{\omega}^{\mathscr{F}^2}\to \boldsymbol{B}_{\omega}^{\mathscr{F}^2}$ виконуються рівності
\begin{equation}\label{eq-3.1}
\varepsilon\chi_{s,q}=\chi_{s,q} \qquad \hbox{i} \qquad \chi_{s,q}\varepsilon=\chi_{s_1,q_1},
\end{equation}
де числа $s_1\in\omega$ i $q_1\in\{0,1\}$ задовольняють умову $(s,s,[q))\varepsilon=(s_1,s_1,[q_1))$. Отже, множина
\begin{equation*}
  \overline{\boldsymbol{End}_{\mathrm{ann}}}(\boldsymbol{B}_{\omega}^{\mathscr{F}^2})=\big\{\epsilon\in \overline{\boldsymbol{End}}(\boldsymbol{B}_{\omega}^{\mathscr{F}^2})\colon \epsilon=\chi_{s,q} \hbox{~для деяких~} s\in\omega, q\in\{0,1\}\big\}
\end{equation*}
збігається з напівгрупою всіх анулюючих ендоморфізмів напівгрупи $\boldsymbol{B}_{\omega}^{\mathscr{F}^2}$. Також з рівності \eqref{eq-3.1} випливає

\begin{theorem}\label{theorem-3.4}
Напівгрупа $\overline{\boldsymbol{End}_{\mathrm{ann}}}(\boldsymbol{B}_{\omega}^{\mathscr{F}^2})$ ізоморфна нескінченній зліченній напівгрупі правих нулів і вона є мінімальним ідеалом напівгрупи  $\overline{\boldsymbol{End}}(\boldsymbol{B}_{\omega}^{\mathscr{F}^2})$.
\end{theorem}

Для довільного натурального числа $k$ означимо відображення $\gamma_k\colon \boldsymbol{B}_{\omega}^{\mathscr{F}^2}\to \boldsymbol{B}_{\omega}^{\mathscr{F}^2}$ i $\delta_k\colon \boldsymbol{B}_{\omega}^{\mathscr{F}^2}\to \boldsymbol{B}_{\omega}^{\mathscr{F}^2}$ за формулами
\begin{equation*}
(i,j,[0))\gamma_k=(i,j,[1))\gamma_k=(ki,kj,[0))
\end{equation*}
та
\begin{equation*}
  (i,j,[0))\delta_k=(ki,kj,[0)) \qquad \hbox{i} \qquad (i,j,[1))\delta_k=(k(i + 1), k(j + 1),[0)),
\end{equation*}
для всіх $i, j\in\omega$, відповідно.

За теоремою~1 з \cite{Gutik-Pozdniakova=2023} для довільного неін'єктивного ендоморфізму $\varepsilon$ напівгрупи $\boldsymbol{B}_{\omega}^{\mathscr{F}^2}$ виконується лише одна з умов:
\begin{itemize}
  \item[$(i)$]   $\varepsilon$~--- анулюючий ендоморфізм;
  \item[$(ii)$]  $\varepsilon=\gamma_k$ для деякого натурального числа $k$;
  \item[$(iii)$] $\varepsilon=\delta_k$ для деякого натурального числа $k$.
\end{itemize}

Тоді з вище наведених результатів і теореми~\ref{theorem-3.1} випливає

\begin{theorem}\label{theorem-3.5}
Для довільного неанулюючого неін'єктивного ендоморфізму $\varepsilon$ напівгрупи $\boldsymbol{B}_{\omega}^{\mathscr{F}^2}$ виконується лише одна з таких умов:
\begin{itemize}
  \item[$(i)$]  існують єдині невід'ємне ціле число $n$ і натуральне число $k$ такі, що $\varepsilon=\gamma_{k}\varpi^n$;
  \item[$(ii)$] існують єдині невід'ємне ціле число $n$ і натуральне число $k$ такі, що $\varepsilon=\delta_{k}\varpi^n$.
\end{itemize}
\end{theorem}
%%%%%%%%%%%%%%%%%%%%%%%%%%%%%%%%%%%%%%%%%%%%%%%%%%%%%%%%%%%%%%%%%%%%%%%%%%%%%%%%%%
%\bigskip

%\section*{{Подяка}}

%Автори висловлюють щиру подяку  рецензентові за цінні поради та зауваження.

%%%%%%%%%%%%%%%%%%%%%%%%%%%%%%%%%%%%%%%%%%%%%%%%%%%%%%%%%%%%%%%%%%%%%%%%%%%%

%\vskip1cm
\end{document}